\documentclass[12pt,twoside]{article}
\usepackage{latexsym, amsthm, amsmath, bbm}
\usepackage{amssymb}
\usepackage{amsfonts}
\usepackage{amstext}
\usepackage{graphicx}
\usepackage{multicol}
\usepackage{subfigure}
\usepackage{textcomp}
\usepackage{enumerate}
\newtheorem{Theorem}{Theorem}

\newtheorem{Proposition}{Proposition}
\newtheorem{Definition}{Definition}
\newtheorem{Corollary}{Corollary}

\theoremstyle{remark}

\theoremstyle{remark}

\theoremstyle{plain}

\theoremstyle{remark}

\newcommand{\R}{{\mathbb{R}}}

\newcommand{\Z}{{\mathbb{Z}}}

\newcommand {\N}{{\mathbb{N}}}

\begin{document}

\title{Edge Boundaries for a Family of Graphs on $\mathbb{Z}^n$}

\author{Ellen Veomett \\
\small Saint Mary's College of California \\
}

\maketitle

\begin{abstract}

We consider the family of graphs whose vertex set is $\Z^n$ where two vertices are connected by an edge when their $\ell_\infty$-distance is 1.  Towards an edge isoperimetric inequality for this graph, we calculate the edge boundary of any finite set $S \subset \Z^n$.  This boundary calculation leads to a desire to show that a set with optimal edge boundary has no ``gaps'' in any direction $\epsilon \in \{-1,0,1\}^n, \epsilon \not=0$.  We show that one can find a set with optimal edge boundary that does not have gaps in any direction $e_i$ (or $-e_i$) where $e_i$ is the standard basis vector.
\end{abstract}

\section{Introduction}

For a metric space $(X,d)$ with a notion of volume and boundary, an isoperimetric inequality gives a lower bound on the boundary of a set of fixed volume.  Ideally, for any fixed volume, it produces a set of that volume with minimal boundary.    The most well-known  isoperimetric inequality states that, in Euclidean space, the unique set of fixed volume with minimal boundary is the Euclidean ball.   

A graph $G = (V,E)$ can be defined as a metric space in the usual way: for $u, v \in V$, 
\begin{equation*}
d(u,v) = \text{the length of the shortest path from } u \text{ to } v.
\end{equation*}
For a graph, an isoperimetric inequality gives a lower bound on the boundary of a set $A \subset V$ of a given size.  The term ``boundary'' here can be interpreted in two standard ways: the vertex boundary or the edge boundary.  The vertex boundary is typically defined as follows: 
\begin{equation*}
\partial A = \{v \in V: d(v,A) \leq 1\}
\end{equation*}
where
\begin{equation*}
d(x,A) = \inf_{a \in A} d(x,a) = \inf_{a \in A}\{\text{the length of the shortest path from } x \text{ to } a\}
\end{equation*}
In words: the vertex boundary of $A$ is the set $A$ itself, along with all of the neighbors of $A$.  The vertex boundary of various graphs has been studied in \cite{MR0200192}, \cite{DiscTor}, \cite{MR1082843},   \cite{MR1612869}, \cite{MR2946103}, and others.  

In this paper, we use another definition for boundary: the edge boundary.  The edge boundary is defined as follows:
\begin{equation*}
\partial_e(A) = \{(x,y) \in E: |A \cap \{x,y\}| = 1\}
\end{equation*}
In words: the edge boundary of $A$ is the set of edges exiting the set $A$.  Many different types of graphs have been studied in terms of the edge isoperimetric question, see for example \cite{MR1137765}, \cite{MR1863367}, \cite{MR1357256}, \cite{MR1755430}, \cite{MR1909858}, \cite{MR2021742}.

Although the vertex and edge isoperimetric inequalities have similar statements, often the resulting optimal sets (and thus the techniques used in their proofs) are quite different.  Indeed, this will be the case for the family of graphs that we consider: $G_n = (\Z^n, E_\infty)$.  For $n \in \N$, the vertices of $G_n$ are the integer points $\Z^n$ in $\R^n$.  The edges $E_\infty$ are between pairs of points whose $\ell_\infty$-distance is 1:
\begin{equation*}
E_\infty = \{(u,v) \in \Z^n: ||u-v||_\infty = 1\}
\end{equation*}
where if $u = (u_1, u_2, \dots, u_n), v = (v_1, v_2, \dots, v_n)$, then
\begin{equation*}
||u-v||_\infty = \max_{i=1, 2, \dots, n} \{|u_i-v_i|\}
\end{equation*}
In  \cite{MR2946103}, the author and A.J. Radcliffe gave the vertex isoperimetric inequality for $(\Z^n, E_\infty)$.  The sets of minimum vertex boundary are nested, and the technique of compression was used to prove this.  Compression relies heavily on the fact that sets of minimum boundary are nested.  Discussions of compression as a technique in discrete isoperimetric problems can be found in \cite{MR2035509},   \cite{MR1444247}, \cite{MR1455181}, and \cite{MR1082842}. 

As was shown in \cite{MR2946103}, sets of size $k^2$ with minimum vertex boundary in $(\Z^2, E_\infty)$ are squares of side length $k$.  In addition, if a set has size which is not a perfect square, then the set achieving optimality will be a rectangular box, or a rectangular box with a strip on one side of the box.  Optimal sets of sizes 1 through 8 in $\Z^2$ are shown in Figure \ref{Boxes}.  

\begin{figure}[htbp]
\centering
\includegraphics[width = 1.2in]{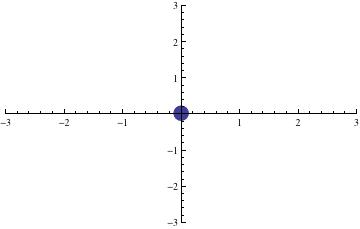} \hspace{.1 cm}
\includegraphics[width = 1.2in]{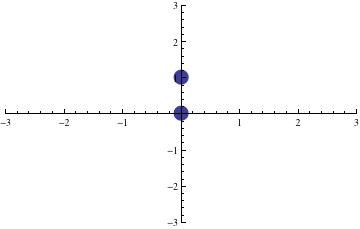}
\includegraphics[width = 1.2in]{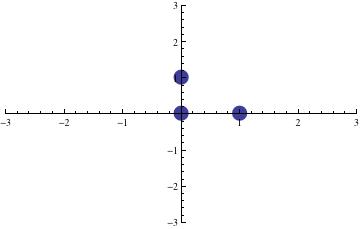}
\includegraphics[width = 1.2in]{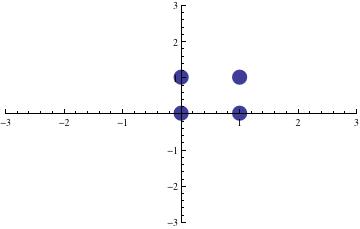}
\includegraphics[width = 1.2in]{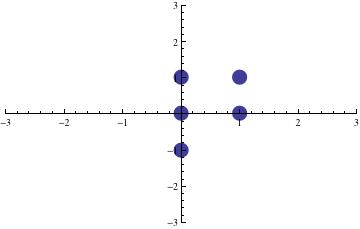}
\includegraphics[width = 1.2in]{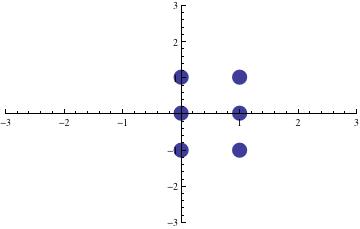}
\includegraphics[width = 1.2in]{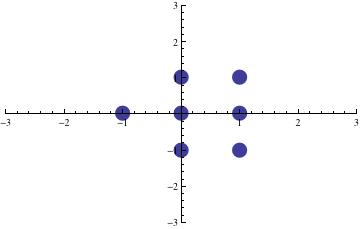}
\includegraphics[width = 1.2in]{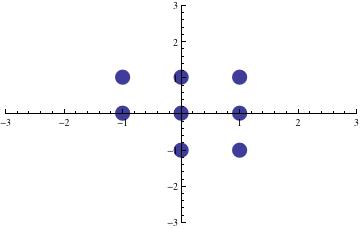}
\caption{Sets of sizes 1 through 8 of minimal vertex boundary in $\Z^2$}
\label{Boxes}
\end{figure}

However, these are not the types of sets which achieve optimal edge boundary.  For example, Figure \ref{Edge1_1} shows a set of size 12 with 36 outgoing edges.  It also has 20 vertex neighbors.  In contrast, Figure \ref{Edge1_2} shows a set of size 12 with 38 outgoing edges and 18 vertex neighbors.  The set in Figure \ref{Edge1_2} has minimal vertex boundary, but in comparison to the set in Figure \ref{Edge1_1}, it cannot have minimal edge boundary.

\begin{figure}[htbp]
\begin{center}
\subfigure[Vertices: 12, Vertex Boundary: 20, Edge Boundary:36]{\label{Edge1_1}\includegraphics[width=2.3 in]{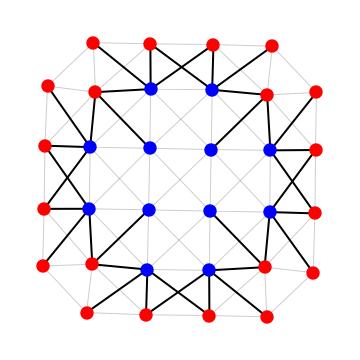}}
\hspace{1 in} \subfigure[Vertices: 12, Vertex Boundary: 18, Edge Boundary:38]{\label{Edge1_2}\includegraphics[width=2.3 in]{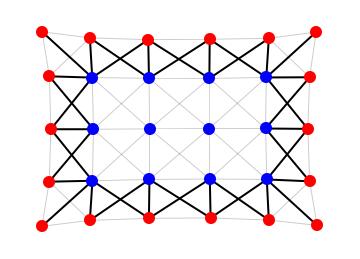}}
\end{center}
\caption{The blue points represent the vertices in the set, the red points are their vertex neighbors in $(\Z^n, E_\infty)$.  Note that \ref{Edge1_1} has a larger vertex boundary than \ref{Edge1_2} but a smaller edge boundary}
\label{Edge1}
\end{figure}

An example in the literature of differing optimal sets when considering vertex boundary versus edge boundary can be found in the graph $(\Z_m^n, E_1)$.  This graph has edge set $\Z_m^n$, where $\Z_m$ denotes the integers modulo $n$.  The edge set consists of all pairs of points whose $\ell_1$ distance is 1:
\begin{equation*}
E_1 = \{(u,v) \in \Z_m^n: ||u-v||_1=1\}
\end{equation*}
where for $u = (u_1, u_2, \dots, u_n)$ and $v=(v_1, v_2, \dots, v_n)$ we have
\begin{equation*}
||u-v||_1 = \sum_{i=1}^n|u_i-v_i|
\end{equation*}
In \cite{DiscTor} Bollob\'as and Leader show that the optimal sets for the vertex isoperimetric problem are nested.  In fact, they correspond to balls using the $\ell_1$-metric.  This is proved using  compression, along with the concepts of fractional systems and symmetrization.  

However, in \cite{MR1137765}, Bollob\'as and Leader show that the optimal sets for the \emph{edge} isoperimetric problem on the same graph are not nested.  They use rectilinear bodies to aid in computing their isoperimetric inequalities.

\section{The Edge Boundary}

We note that we use the \emph{unordered pair} notation for the edges $E_\infty$.  That is, if $(u,v) \in E_\infty$, then we consider $(v,u) \in E_\infty$ and $(u,v) = (v,u)$.

It is not too hard to calculate the edge boundary for a general set $S \subset \Z^n$ in the graph $(\Z^n, E_\infty)$.  First we require a couple of definitions.

For $S \subset \Z^n$ and $\epsilon \in \{-1, 0, 1\}^n$, let $P_\epsilon(S)$  be the projection of $S$ onto $\epsilon^\perp$.  That is,
\begin{equation*}
P_\epsilon(S) = \left\{u-\frac{\left<u,\epsilon\right>}{||\epsilon||_2}\epsilon: u \in S\right\}
\end{equation*}
where for $u = (u_1, u_2, \dots, u_n)$ and $\epsilon = (\epsilon_1, \epsilon_2, \dots, \epsilon_n)$, 
\begin{align*}
\left<u,\epsilon\right> &= \sum_{i=1}^n u_i \epsilon_i \\
||\epsilon||_2 &= \sqrt{\sum_{i=1}^n\epsilon_i^2}
\end{align*}

We also need the following:

\begin{Definition}
Let $S \subset \Z^n$ be finite.  For $\epsilon \in \{-1, 0, 1\}^n$ with $\epsilon \not=0$, we define 
\begin{equation*}
\text{gap}_\epsilon(S) = \{x \in \Z^n: x-\epsilon \in S, x \not\in S, \text{ and } x+b\epsilon \in S \text{ for some }b \geq 1\}
\end{equation*}
Thus, one can think of a point $x \in \text{gap}_\epsilon(S)$ as the first vertex in $\Z^n$ which indicates a gap in $S$ in the line through $x$ in the direction of $\epsilon$.
\end{Definition}

We have the following:

\begin{Theorem}\label{GapThm}
Let $S \subset \Z^n$ be a finite set.  Then
\begin{equation}\label{eqn1}
|\partial_e (S)| = \sum_{\epsilon \in \{-1,0,1\}^n, \epsilon \not=0}\left( |P_\epsilon(S)| +|\text{gap}_\epsilon(S)|\right)
\end{equation}
\end{Theorem}

\begin{proof}
We proceed by induction on $|S|$.  If $|S| =1$, then $\text{gap}_\epsilon(S) = \emptyset$ for each $\epsilon \in \{-1, 0, 1 \}^n, \epsilon \not=0$.  We can also see that if $S = \{u\}$, then
\begin{equation*}
\partial_\epsilon( S) = \{(u, u+\epsilon): \epsilon \in \{-1,0,1\}^n: \epsilon \not= 0\}
\end{equation*}
We can also see that in this case, 
\begin{equation*}
|P_\epsilon(S)| = 1
\end{equation*}
for each $\epsilon \in \{-1, 0, 1\}^n, \epsilon \not= 0$.  Thus, we have 
\begin{equation*}
|\partial_e (S)| = \sum_{\epsilon \in \{-1,0,1\}^n, \epsilon \not=\vec{0}}\left( |P_\epsilon(S)| +|\text{gap}_\epsilon(S)|\right)
\end{equation*}
if $|S|=1$.

Now suppose that $|S|>1$.  Fix $u \in S$.  By induction,
\begin{equation*}
|\partial_e (S\backslash\{u\})| = \sum_{\epsilon \in \{-1,0,1\}^n, \epsilon \not=0}\left( |P_\epsilon(S\backslash\{u\})| +  |\text{gap}_\epsilon(S\backslash\{u\})|\right)
\end{equation*}
Consider what $u$ contributes to the edge boundary of $S$.  Note that each $\epsilon \in \{-1,0,1\}^n, \epsilon \not=0$ can be uniquely paired with $-\epsilon \in \{-1,0,1\}^n$.  We have three cases:

\noindent \underline{Case 1: Both $u+\epsilon$ and $u-\epsilon$ are in $S$}  In this case, 
\begin{align*}
(u+\epsilon, u)  \in \partial_e(S\backslash \{u\}) & \quad\quad \quad   (u-\epsilon, u)  \in \partial_e(S\backslash \{u\})  \\
(u+\epsilon, u)  \not\in \partial_e(S) &  \quad\quad \quad   (u-\epsilon, u)  \not\in \partial_e(S)
\end{align*}
and
\begin{equation*}
\text{gap}_\epsilon(S) =  \text{gap}_\epsilon(S\backslash\{u\})\backslash\{u\} \quad  \quad \quad \text{gap}_{-\epsilon}(S) = \text{gap}_{-\epsilon}(S\backslash\{u\})\backslash\{u\}
\end{equation*}
Thus we can see that both the left and right hand sides of equation \eqref{eqn1} go down by 2 corresponding to edges $(u, u+\epsilon), (u, u-\epsilon)$ when $u$ is added back to $S$.

\noindent \underline{Case 2: Exactly one of  $u+\epsilon$ or $u-\epsilon$ is in $S$}  Here, without loss of generality, assume that $u-\epsilon \in S$.  Then
\begin{align*}
(u-\epsilon, u)  \in \partial_e(S\backslash \{u\}) & \quad\quad \quad   (u+\epsilon, u)  \not\in \partial_e(S\backslash \{u\})  \\
(u-\epsilon, u)  \not\in \partial_e(S) &  \quad\quad \quad   (u+\epsilon, u)  \in \partial_e(S)
\end{align*}
and
\begin{equation*}
\text{gap}_\epsilon(S) =  \text{gap}_\epsilon(S\backslash\{u\})\backslash\{u\} \cup \{u+\epsilon\} \quad  \quad \quad \text{gap}_{-\epsilon}(S) = \text{gap}_{-\epsilon}(S\backslash\{u\})
\end{equation*}

Thus we can see that both the left and right hand sides of equation \eqref{eqn1} do not change corresponding to edges $(u, u+\epsilon), (u, u-\epsilon)$ when $u$ is added back to $S$.

\noindent \underline{Case 3: Neither $u+\epsilon$ nor $u-\epsilon$ are in $S$}

In this case, 
\begin{align*}
(u+\epsilon, u)  \not\in \partial_e(S\backslash \{u\}) & \quad\quad \quad   (u-\epsilon, u)  \not\in \partial_e(S\backslash \{u\})  \\
(u+\epsilon, u)  \in \partial_e(S) &  \quad\quad \quad   (u-\epsilon, u)  \in \partial_e(S)
\end{align*}
and
\begin{equation*}
\text{gap}_\epsilon(S) =  \text{gap}_\epsilon(S\backslash\{u\})\cup \{u+\epsilon\} \quad  \quad \quad \text{gap}_{-\epsilon}(S) = \text{gap}_{-\epsilon}(S\backslash\{u\})\cup \{u-\epsilon\}
\end{equation*}

Thus we can see that both the left and right hand sides of equation \eqref{eqn1} go up by 2 corresponding to edges $(u, u+\epsilon), (u, u-\epsilon)$ when $u$ is added back to $S$.

Since $\epsilon$ was arbitrary, we can see that all of the changes between $\partial_e(S \backslash\{u\})$ and $\partial_e(S)$ are balanced out by changes in the corresponding gaps.  Thus, we have 
\begin{equation}\label{eqn1}
|\partial_e (S)| = \sum_{\epsilon \in \{-1,0,1\}^n, \epsilon \not=0}\left( |P_\epsilon(S)| +|\text{gap}_\epsilon(S)|\right)
\end{equation}

\end{proof}

Theorem \ref{GapThm} clearly has the following corollary:

\begin{Corollary}
Let $S \subset \Z^n$ be a finite set such that $\text{gap}_\epsilon(S) = \emptyset$ for each $\epsilon \in \{-1,0,1\}^n, \epsilon \not=0$.  Then
\begin{equation}\label{eqn1}
|\partial_e (S)| = \sum_{\epsilon \in \{-1,0,1\}^n, \epsilon \not=0}|P_\epsilon(S)| 
\end{equation}

\end{Corollary}

which is a much more satisfying result, as it only involves $n-1$-dimensional projections of $S$, and is a nice counterpoint to the vertex boundary calculations in \cite{MR2946103}.  This leads to the natural desire to show that it is possible to ``squish'' any set $S \subset \Z^n$ to form a new set $S_0$ such that $|S| = |S_0|$, $\partial_e(S)\geq \partial_e(S_0)$,  and $S_0$ has no gaps (that is, $\text{gap}_\epsilon(S_0) = \emptyset$ for each $\epsilon \in \{-1,0,1\}^n, \epsilon \not=0$).  

In the following section, we show how we can ``squish'' set $S$ into set $S_i$ so that $|S| = |S_i|, \partial_e(S) \geq \partial_e(S_i)$, and
\begin{equation*}
\text{gap}_{e_i}(S_i) = \emptyset
\end{equation*}
where $e_i \in \{-1,0,1\}^n$ is the $i$th standard basis vector.

\section{Central Compression}

The following notation and definitions are similar to those  in \cite{MR2946103}. 
For simplicity, we introduce the following notation: for a real-valued vector $p = (p_1, p_2, \dots, p_n) \in \R^n$, and $x \in \R$, we define
\begin{equation*}
(p, x \rightarrow i) = (p_1, p_2, \dots, p_{i-1}, x, p_i, p_{i+1}, \dots, p_n) \in \R^{n+1}
\end{equation*}
In words, $(p, x \rightarrow i)$ is the vector that results when placing $x$ in the $i$th coordinate of $p$ and shifting the $i$th through $n$th coordinates of $p$ to the right.

\begin{Definition}
We say that a set $S \subset \Z^n$ is \emph{centrally compressed} in the $i$-th coordinate ($1 \leq i \leq n$) with respect to $p \in \Z^{n-1}$ if the set
\begin{equation*}
\{x \in \Z: (p, x \rightarrow i) \in S\}
\end{equation*}
is either empty or of one of the following two forms:
\begin{align*}
\{x: -a  \leq x &\leq a \text{ for } a \in \N\} \\
& \text{OR} \\
\{x: -a \leq x &\leq a+1 \text{ for } a \in \N\}
\end{align*}

\end{Definition}

This definition allows us to define the $i$th central compression of a set:

\begin{Definition}
Let $S \subset \Z^n$.  For $1 \leq i \leq n$, we define $S_i$ to be the $i$th central compression of $S$  by specifying its 1-dimensional sections in the $i$th coordinate.  Specifically, 
\begin{enumerate}
\item For each $p \in \Z^{n-1}$, 
\begin{equation*}
|\{x \in \Z: (p, x \rightarrow i) \in S_i\}| 
= |\{x \in \Z: (p, x \rightarrow i) \in S\}| 
\end{equation*}
\item $S_i$ is centrally compressed in the $i$th coordinate with respect to $p$ for each $p \in \Z^{n-1}$.
\end{enumerate}

\end{Definition}

In words: after fixing a coordinate $i \in \{1, 2, \dots, n\}$, we consider all lines in $\Z^n$ where only the $i$th coordinate varies, and we intersect those lines with $S$.  Each of the points in those intersections are moved along the line so that they are a segment centered around 0.  The result is $S_i$.  

The following Proposition shows how we can ``squish'' set $S$ into set $S_i$ so that $|S| = |S_i|, \partial_e(S) \geq \partial_e(S_i)$, and
\begin{equation*}
\text{gap}_{e_i}(S_i) = \emptyset
\end{equation*}
where $e_i \in \{-1,0,1\}^n$ is the $i$th standard basis vector.

\begin{Proposition}\label{Compression1}
Suppose that $S \subset \Z^n$.    For $1 \leq i \leq n$, let $S_i$ be the $i$th central compression of $S$.  Then 
\begin{equation*}
\left|\partial_e S_i\right| \leq \left| \partial_e S \right|
\end{equation*}
\end{Proposition}

\begin{proof}
Suppose that $S \subset \Z^n$ and fix $i \in \{1, 2, \dots, n\}$.  First we note that we can count the edge boundary of $S$ by partitioning the outgoing edges of $S$ into the sets of edges coming from each 1-dimensional $i$-section of $S$.  Specifically for $p \in \Z^{n-1}$, let
\begin{equation*}
\partial_e(S,p) = \{(u,v) \in E_\infty: u \in S, v \in \Z^n \backslash S \text{ and } u = (p, x \rightarrow i) \text{ for some } x \in \Z\}
\end{equation*}
Then we have
\begin{equation*}
\partial_e S = \bigcup_{p \in \Z^{n-1}} \partial_e (S,p)
\end{equation*}
and the above union is disjoint.   

We can partition these even further, based on which 1-dimensional section the vertex which is not in $S$ lies.  That is,  if $(u,v) \in E_\infty$ with $u \in S$, $v \in \Z^n \backslash S$, and $u = (p, x \rightarrow i)$ for  $x \in \Z$, then we must have
\begin{equation*}
v = (p +\epsilon, y \rightarrow j)
\end{equation*}
for some $\epsilon \in \{-1,0,1\}^{n-1}$ and $y \in \Z$ (specifically, $y \in \{x-1, x, x+1\}$).  Let
\begin{multline*}
\partial_e (S,p,\epsilon) =  \{(u,v) \in E_\infty: u \in S, v \in \Z^n \backslash S,  u = (p, x \rightarrow i) \text{ for some } x \in \Z,   \\
\text{ and } v = (p+\epsilon, y \rightarrow i) \text{ for some }y \in \Z\}.
\end{multline*}
Then
\begin{equation*}
\partial_e(S,p) = \bigcup_{\epsilon \in \{-1,0,1\}^{n-1}} \partial_e(S, p,\epsilon)
\end{equation*}
and the union is disjoint.  Thus, we have

\begin{equation*}
\partial_e S = \bigcup_{p \in \Z^{n-1}} \bigcup_{\epsilon \in \{-1,0,1\}^{n-1}} \partial_e(S,p,\epsilon)
\end{equation*}
and the above unions are disjoint, so that
\begin{equation*}
\left|  \partial_e S \right| = \sum_{p \in \Z^{n-1}} \sum_{\epsilon \in \{-1,0,1\}^{n-1}} \left|  \partial_e(S,p,\epsilon)  \right|
\end{equation*}
Similarly,
\begin{equation*}
\left|  \partial_e S_i \right| = \sum_{p \in \Z^{n-1}} \sum_{\epsilon \in \{-1,0,1\}^{n-1}} \left|  \partial_e(S_i,p,\epsilon)  \right|
\end{equation*}
We will show that $\left| \partial_e S_i \right| \leq \left| \partial_e S \right|$ by showing that
\begin{equation*}
 \left|  \partial_e(S_i,p,\epsilon)  \right| \leq  \left|  \partial_e(S,p,\epsilon)  \right|
\end{equation*}
for each $p \in \Z^{n-1}$ and $\epsilon \in \{-1,0,1\}^{n-1}$.

It is straightforward to see that for $\vec{0} \in \{-1, 0, 1\}^{n-1}$, 
\begin{align*}
 \left|  \partial_e(S_i,p,\vec{0}) \right|&= 2 \\
\left|  \partial_e(S,p,\vec{0})  \right| & \geq 2
\end{align*}
so that
\begin{equation*}
 \left|  \partial_e(S_i,p,\vec{0})  \right| \leq  \left|  \partial_e(S,p,\vec{0})  \right|
\end{equation*}

Thus, we can now consider a fixed $p$ and fixed $\epsilon \not=0$.  Let
\begin{align*}
\ell &= \{(p, x\rightarrow i)\in S : x \in \Z \} \\
n &= \{(p+\epsilon, y \rightarrow i)\in S: y \in \Z\}
\end{align*}
be the lines of vertices in $S$ corresponding to fixing all entries except the $i$th by the entries in vectors $p$ and $p+\epsilon$ respectively.  Note that we can visualize these lines as integer points in the plane, with the lines parallel to the $x$-axis.

For all of our visualizations, the upper line will denote $n$, the lower line $\ell$.  Open circles are vertices in $n$ or $\ell$, filled in circles are vertices which are not in $n$ or $\ell$.  The solid lines are edges which are definitely in $\partial_e(S,p,\epsilon)$; the dotted lines are edges which may be in $\partial_e(S,P,\epsilon)$ (depending on whether particular vertices are in $\ell$ and $n$).  See Figure \ref{VisualizationP}.

\begin{figure}[h]
\begin{center}
\includegraphics[width=4 in]{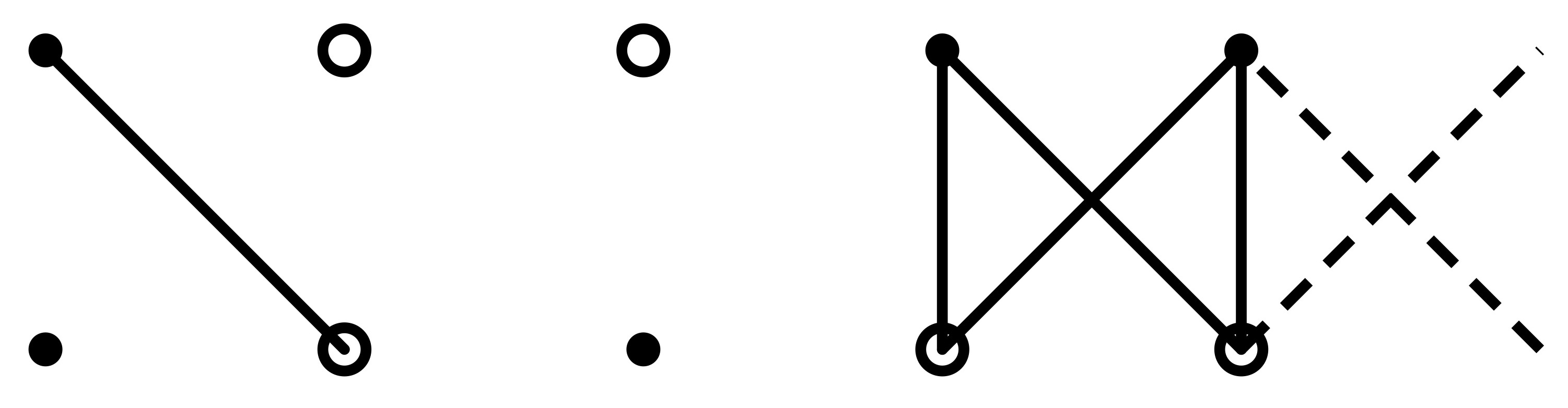}
\end{center} 
\caption{Sample visualization of $n$ and $\ell$.}
\label{VisualizationP}
\end{figure}

We say that there is a ``gap'' in the line $\ell$ at $b$ if there exist $a<b<c$ such that $(p, b \rightarrow i) \not\in \ell$, but $(p, a \rightarrow i) \in \ell$ and $(p, c \rightarrow i) \in \ell$.  See Figure \ref{NewGapP}.

\begin{figure}[h]
\begin{center}
\includegraphics[width=4 in]{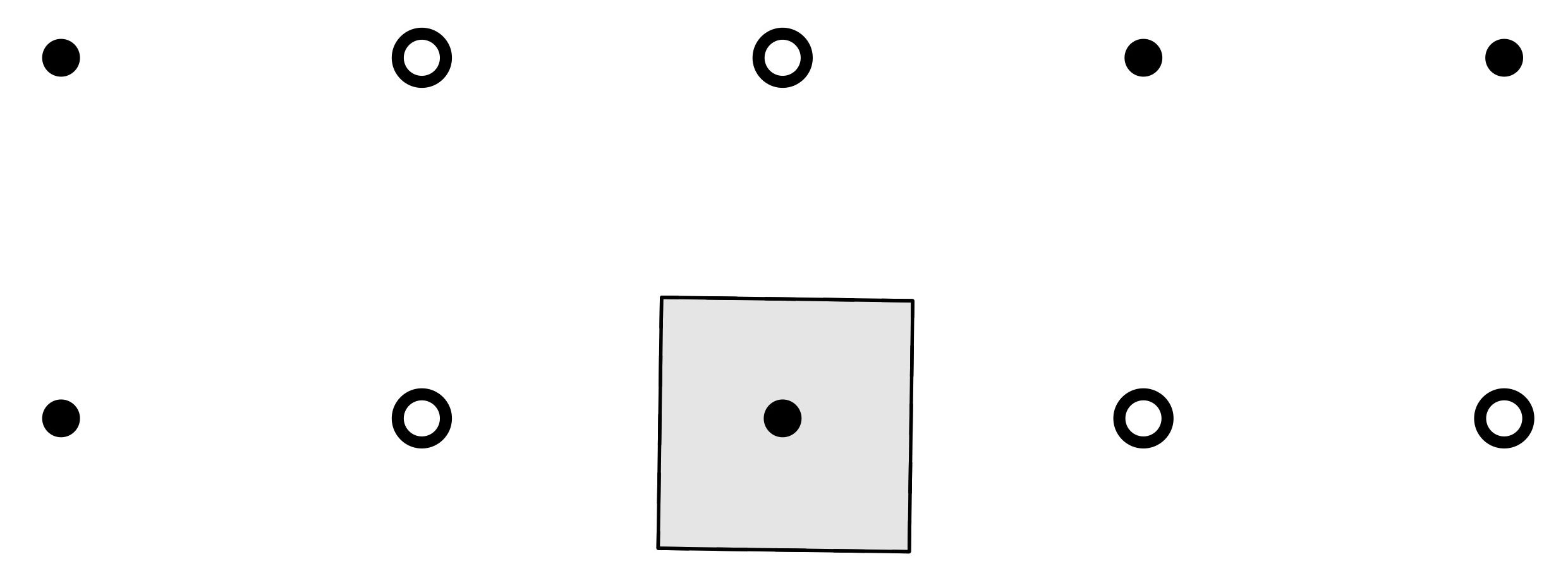}
\end{center} 
\caption{The gap in $\ell$ is highlighted.}
\label{NewGapP}
\end{figure}
We can similarly define a gap in $n$.

We claim that we can rearrange the vertices within $\ell$ and $n$ by perhaps changing the $i$th coordinate of some vertices to make new lines of vertices $\ell_0$ and $n_0$ with no gaps, such that $\partial_e(\ell \cup n, p, \epsilon) \geq \partial_e(\ell_0 \cup n_0,p,\epsilon)$.  We do this in a sequence of steps.

First note that if both $\ell$ and $n$ have a gap at $b$, then the vertices of $\ell$ to the left of the gap and the vertices of $n$ to the left of the gap can all be shifted to the right by 1 without increasing $\partial_e(\ell \cup n, p, \epsilon)$; only possibly decreasing $\partial_e(\ell \cup n, p, \epsilon)$ (See Figure \ref{SampleShift}).

\begin{figure}[htbp]
\begin{center}
\includegraphics[width=2 in]{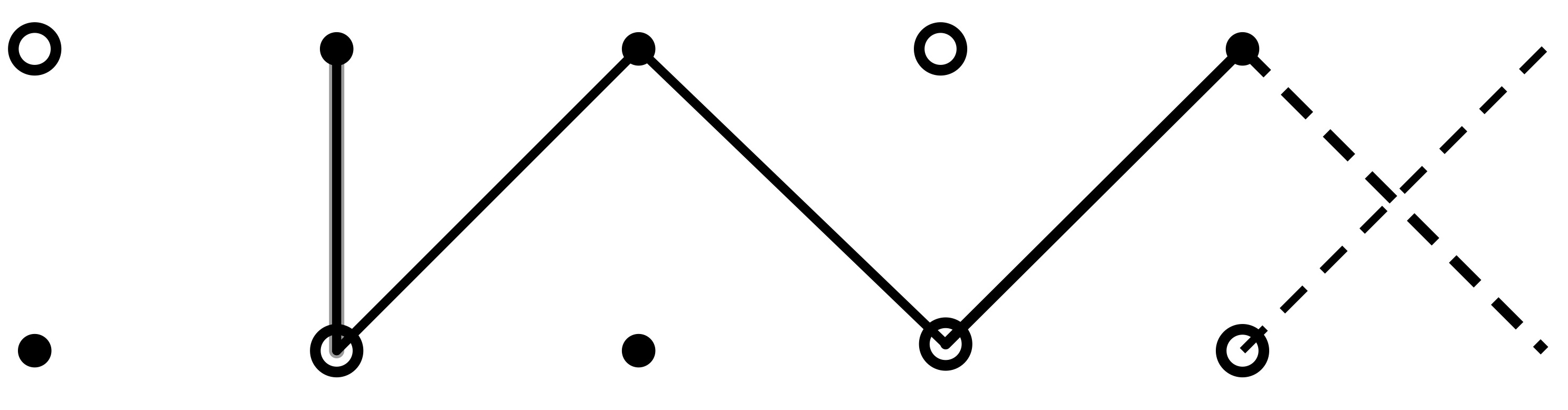}
\hspace{.5 cm}
\includegraphics[width=.5 in]{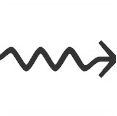}
\hspace{.5 cm}
\includegraphics[width=2 in]{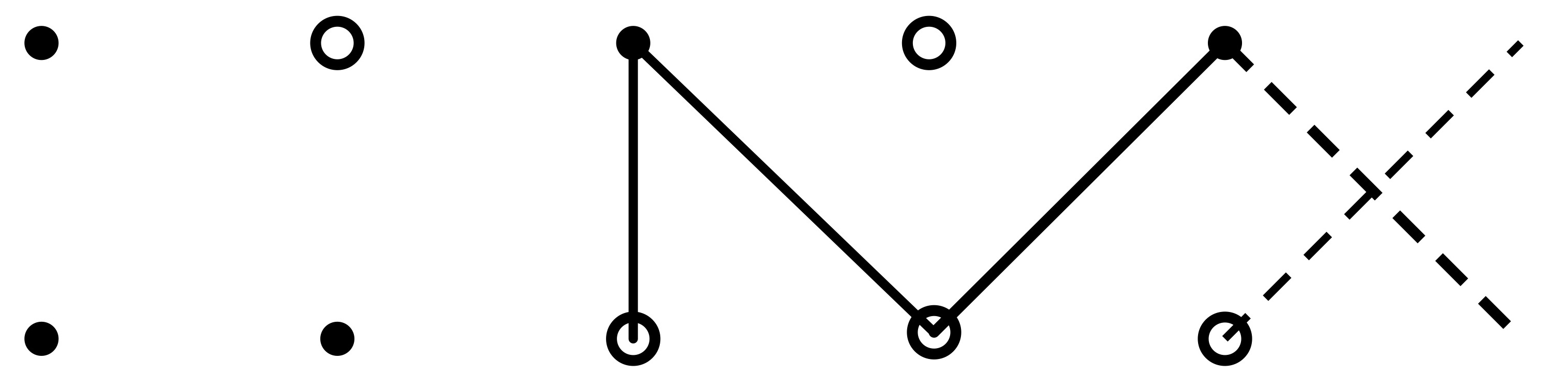}
\end{center} 
\caption{The parallel gap is eliminated by shifting vertices to the right.  In this case, the edge boundary decreases.}
\label{SampleShift}
\end{figure}

Thus, we can assume that there are no such parallel gaps.

Let $\ell^i$ be the minimum value of $x$ for any point $(p,x \rightarrow i) \in \ell$ and let $n^i$ be the minimum value of $y$ for any point $(p,y\rightarrow i) \in n$.

We now split this problem into 6 cases:
\begin{enumerate}
\item $\ell^i = n^i$, there is a gap in $\ell$ at $b$, and there is no gap in either $\ell$ or $n$ for any $a<b$.
\item $\ell^i = n^i$, there is a gap in $n$ at $b$, and there is no gap in either $\ell$ or $n$ for any $a<b$.
\item $\ell^i = n^i+1$
\item $\ell^i = n^i-1$
\item $\ell^i = n^i+c$ where $c \geq 2$
\item $\ell^i = n^i-c$ where $c \geq 2$.
\end{enumerate}
In each of these cases, we can see that the number of edges going from a vertex in $\ell$ to a vertex not in $n$ (i.e. the number of edges in $\partial_e(\ell \cup n, p, \epsilon)$) either does not change or decreases by shifting some vertices in $n$ to the right or shifting some vertices in $\ell$ to the right.  This can more easily be seen through pictures:

\begin{enumerate}
\item $\ell^i = n^i$, there is a gap in $\ell$ at $b$, and there is no gap in either $\ell$ or $n$ for any $a<b$.
\begin{figure}[htbp]
\begin{center}
\includegraphics[width=1.3 in]{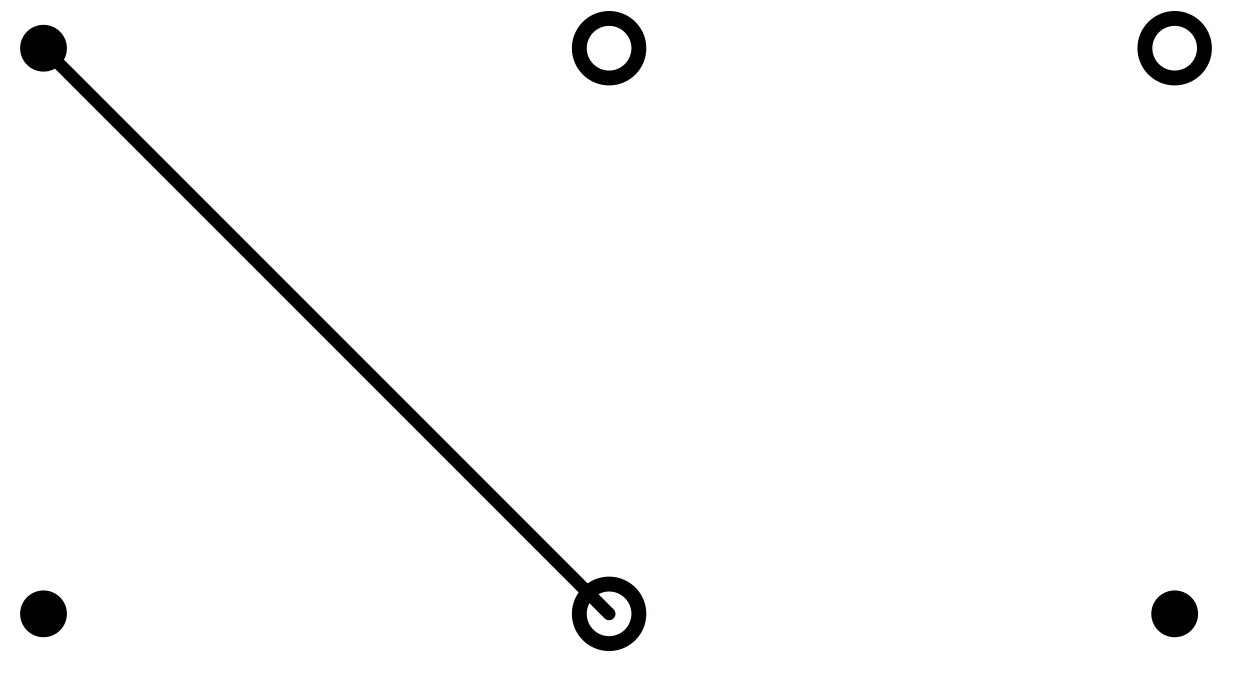}
\hspace{.7 cm}
\includegraphics[width=.5 in]{rsarrow.jpg}
\hspace{.7 cm}
\includegraphics[width=1.9 in]{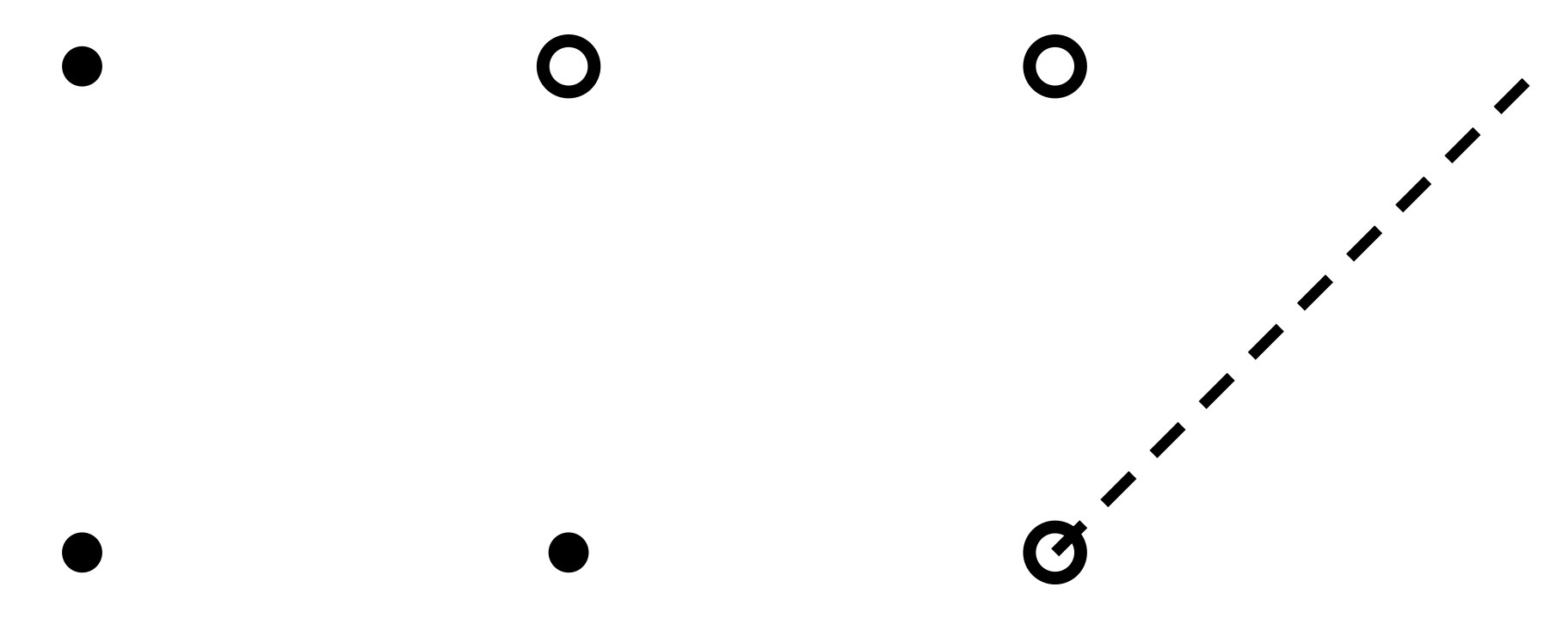}
\end{center} 
\caption{Vertices in $\ell$ to the left of the gap are shifted right by 1.}
\end{figure}

\item $\ell^i = n^i$, there is a gap in $n$ at $b$, and there is no gap in either $\ell$ or $n$ for any $a<b$.

\begin{figure}[htbp]
\begin{center}
\includegraphics[width=1.8 in]{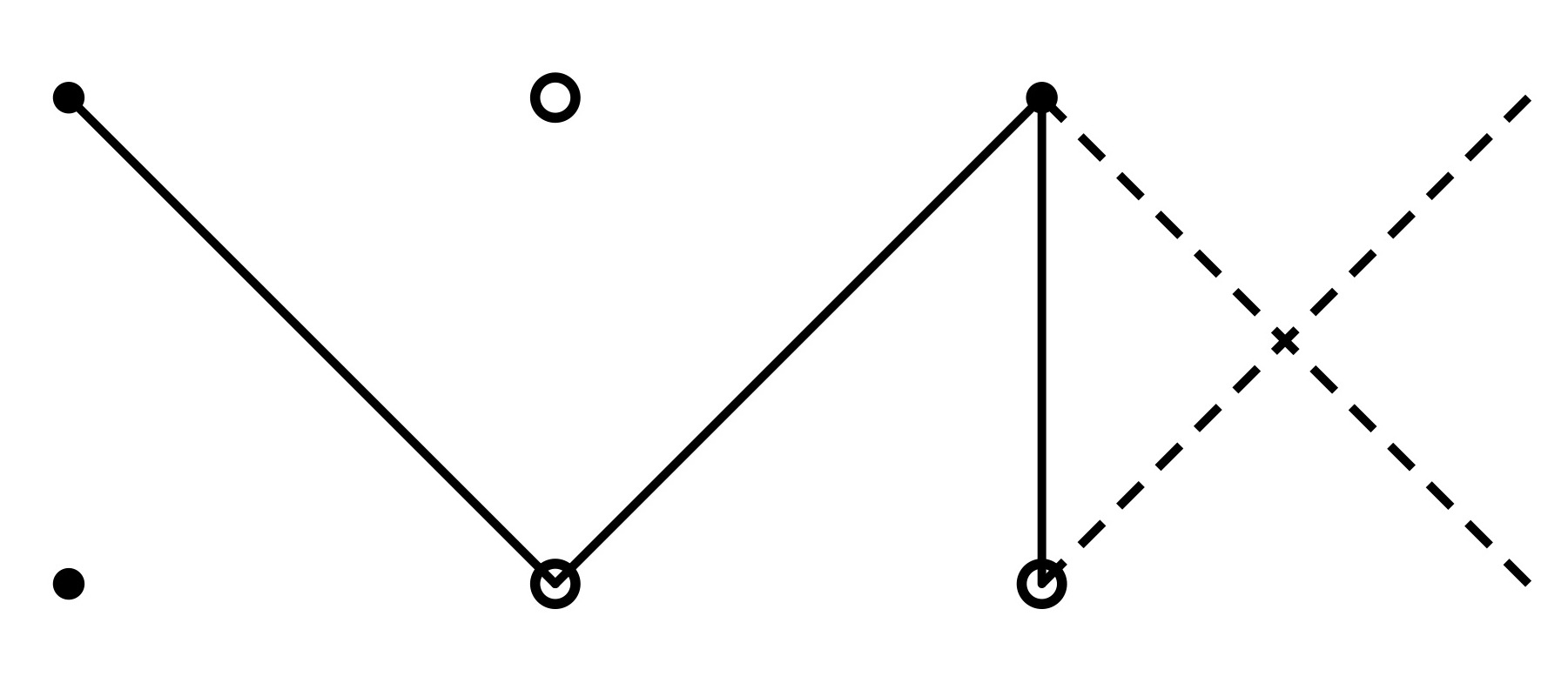}
\hspace{.7 cm}
\includegraphics[width=.5 in]{rsarrow.jpg}
\hspace{.7 cm}
\includegraphics[width=1.8 in]{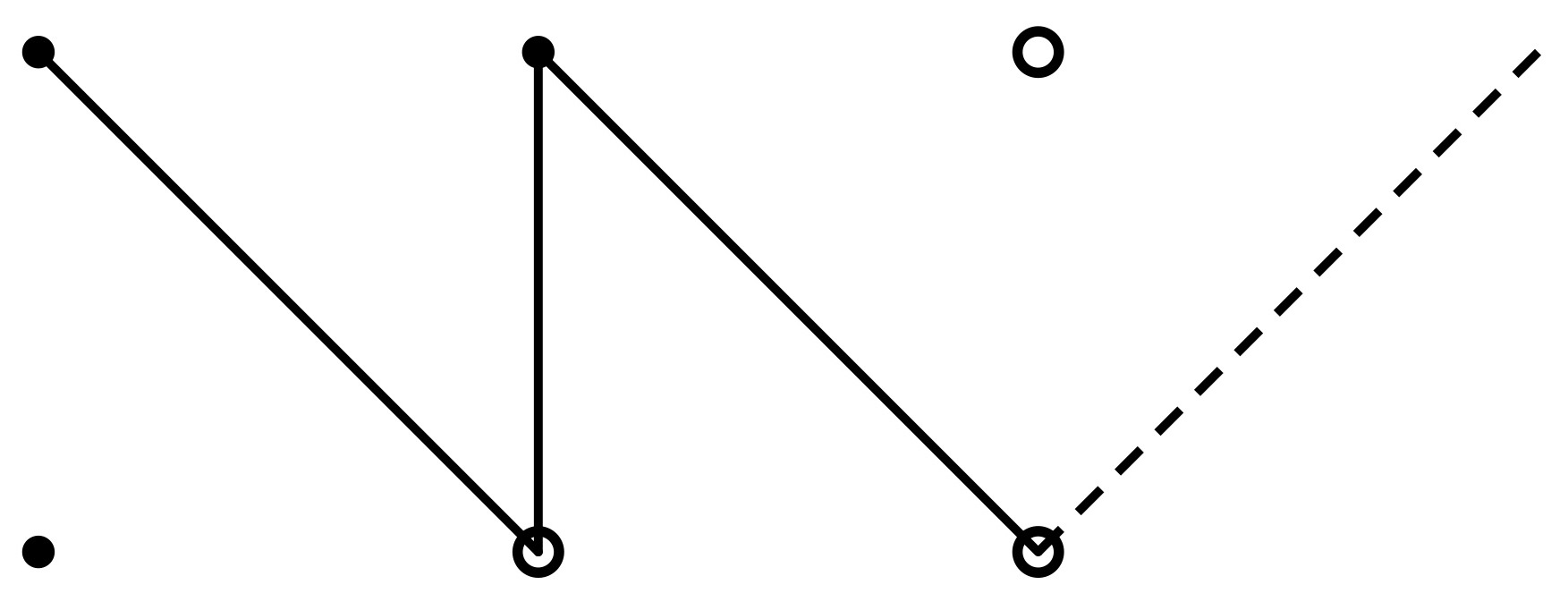}
\end{center} 
\caption{Vertices in $n$ to the left of the gap are shifted right by 1.}
\end{figure}

\item $\ell^i = n^i+1$

\begin{figure}[htbp]
\begin{center}
\includegraphics[width=2.2 in]{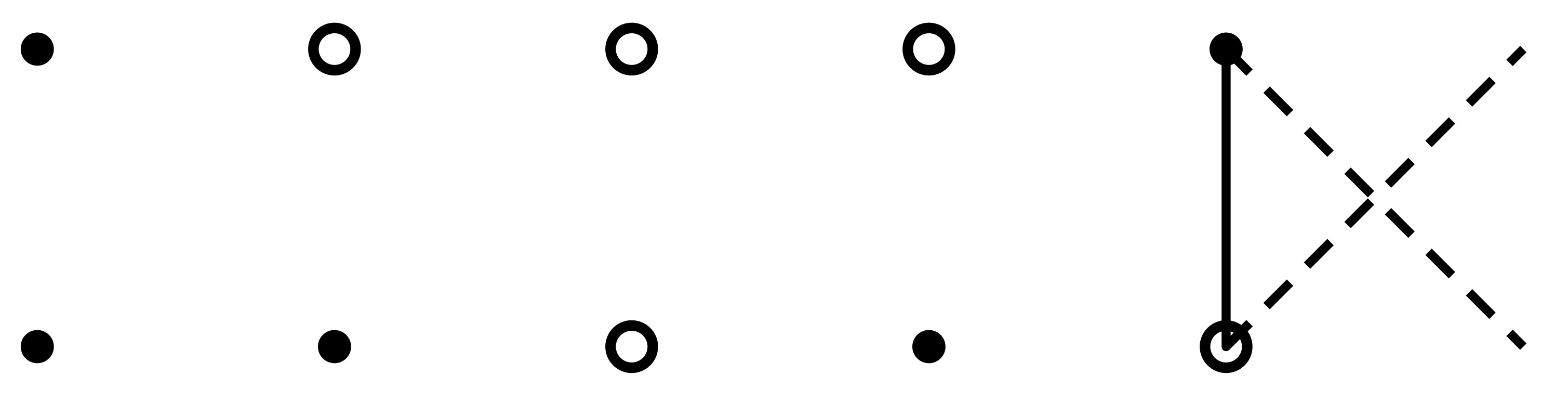}
\hspace{.7 cm}
\includegraphics[width=.5 in]{rsarrow.jpg}
\hspace{.7 cm}
\includegraphics[width=2.2 in]{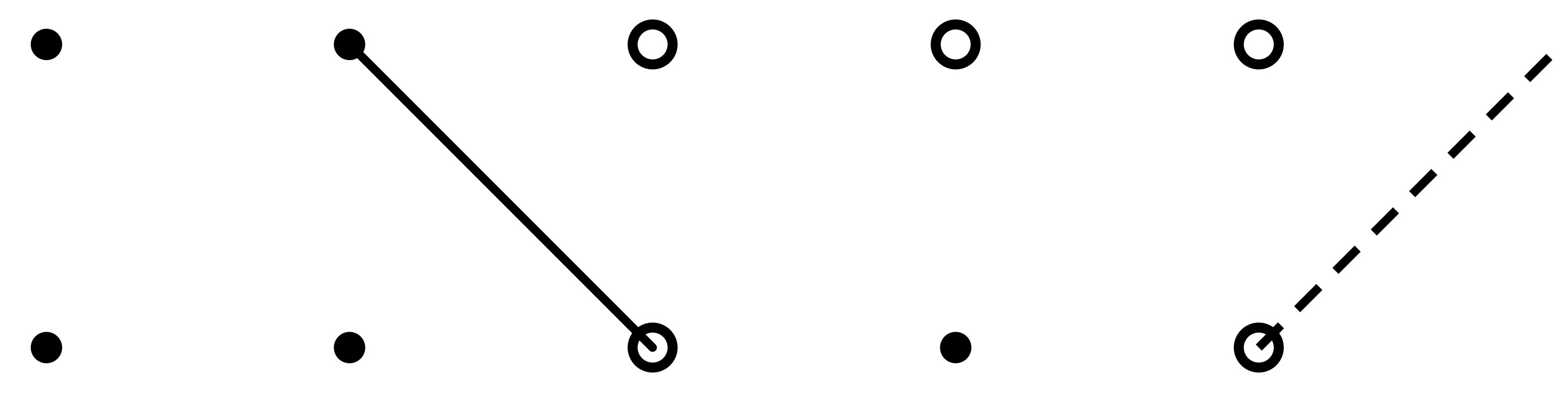}
\end{center} 
\caption{Vertices in $n$ to the left of the first gap (or first vertex in $n$ not in $S$) are shifted right by 1.}
\end{figure}

\pagebreak

\item $\ell^i = n^i-1$

\begin{figure}[htbp]
\begin{center}
\includegraphics[width=2.1 in]{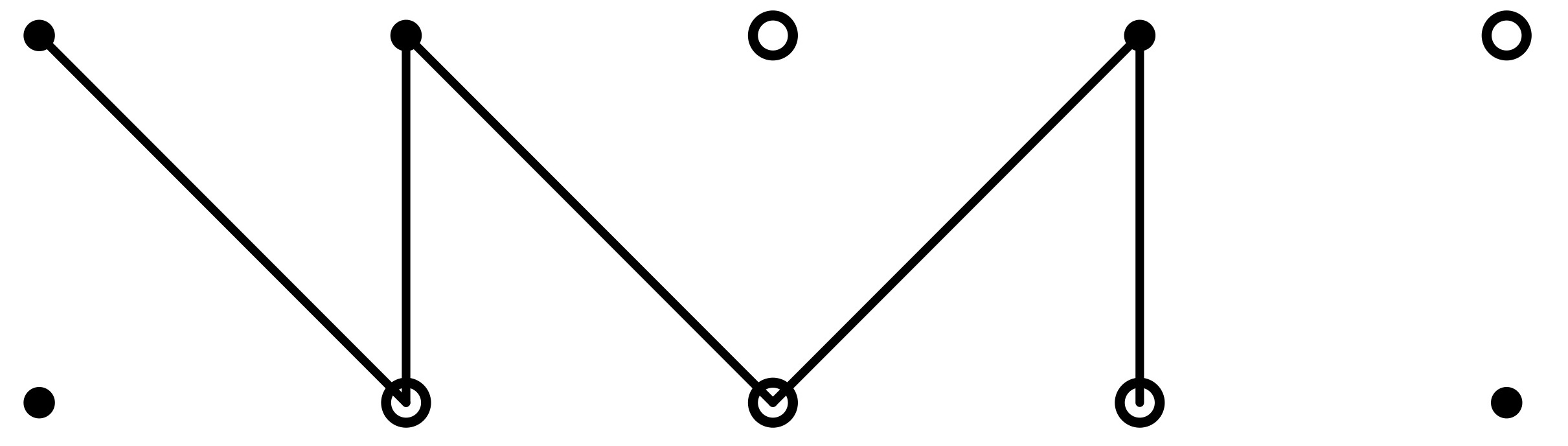}
\hspace{.7 cm}
\includegraphics[width=.5 in]{rsarrow.jpg}
\hspace{.7 cm}
\includegraphics[width=2.3 in]{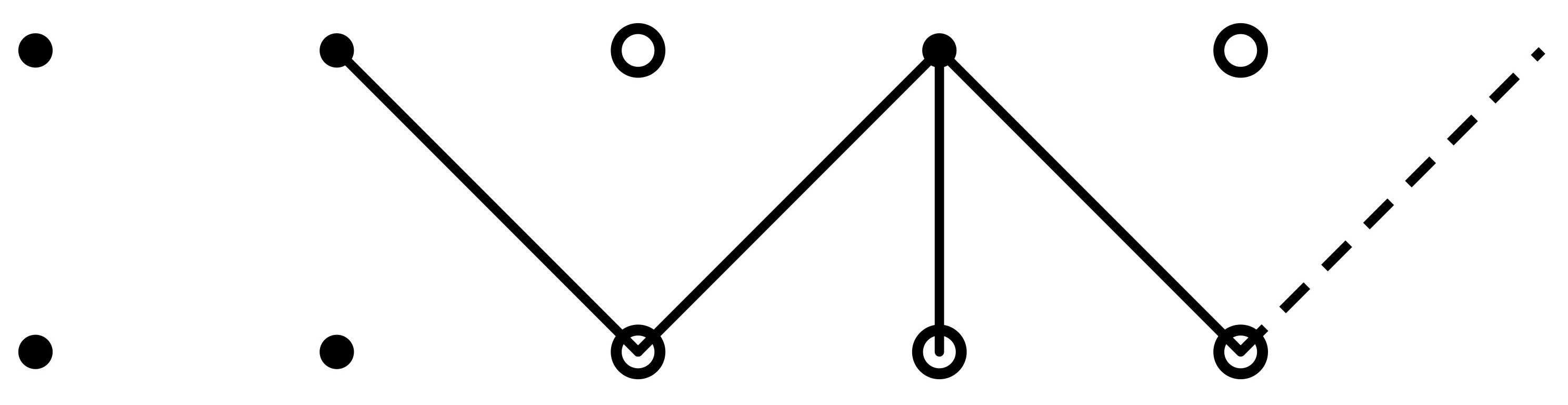}
\end{center} 
\caption{Vertices in $\ell$ to the left of the first gap (or first vertex in $\ell$ not in $S$) are shifted right by 1.}
\end{figure}

\item $\ell^i = n^i+c$ where $c \geq 2$

\begin{figure}[htbp]
\begin{center}
\includegraphics[width=2.1 in]{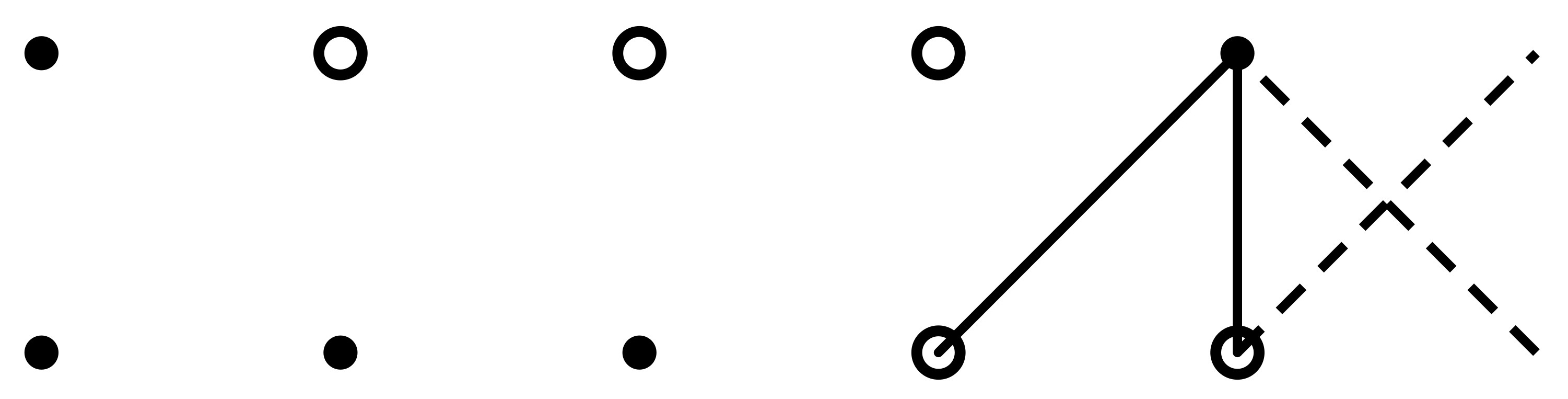}
\hspace{.7 cm}
\includegraphics[width=.5 in]{rsarrow.jpg}
\hspace{.7 cm}
\includegraphics[width=2.2 in]{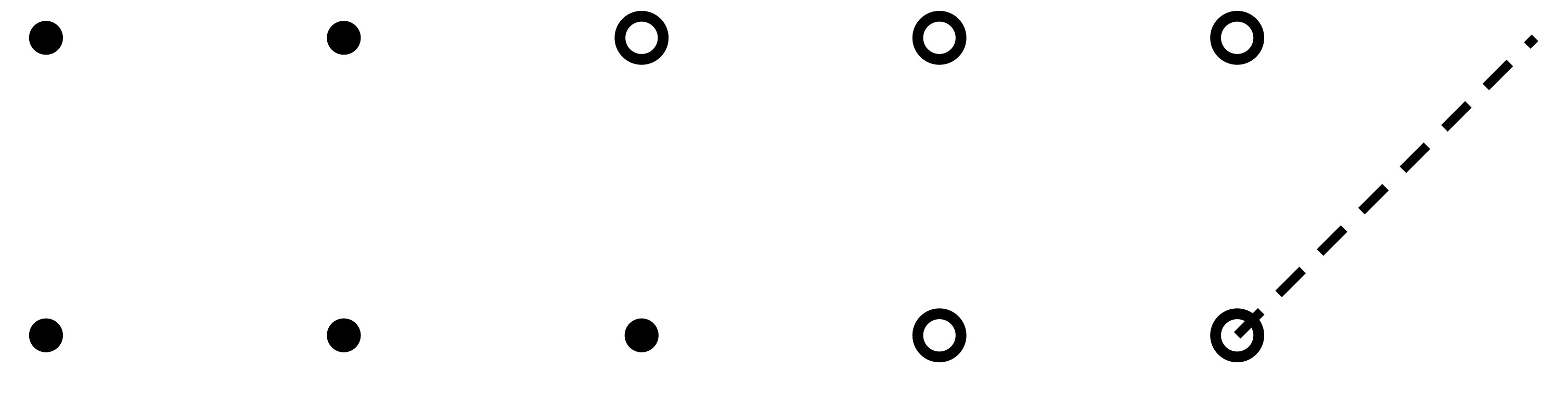}
\end{center} 
\caption{Vertices in $n$ to the left of the first gap (or first vertex in $n$ not in $S$) are shifted right by 1.}
\end{figure}

\item $\ell^i = n^i-c$ where $c \geq 2$.

\begin{figure}[htbp]
\begin{center}
\includegraphics[width=2.1 in]{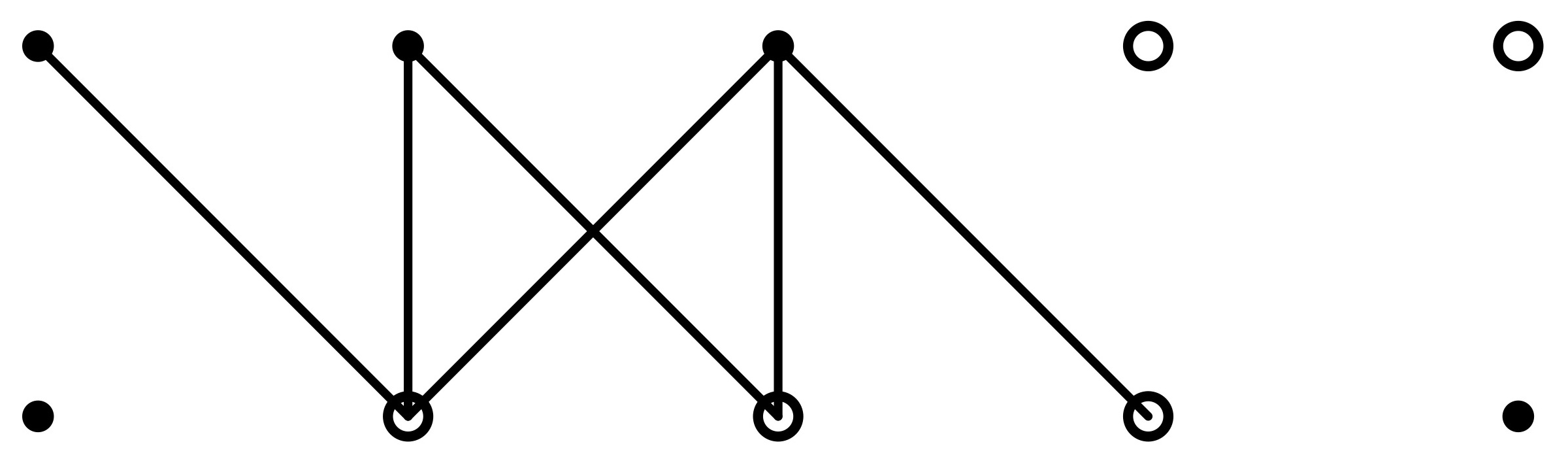}
\hspace{.7 cm}
\includegraphics[width=.5 in]{rsarrow.jpg}
\hspace{.7 cm}
\includegraphics[width=2.2 in]{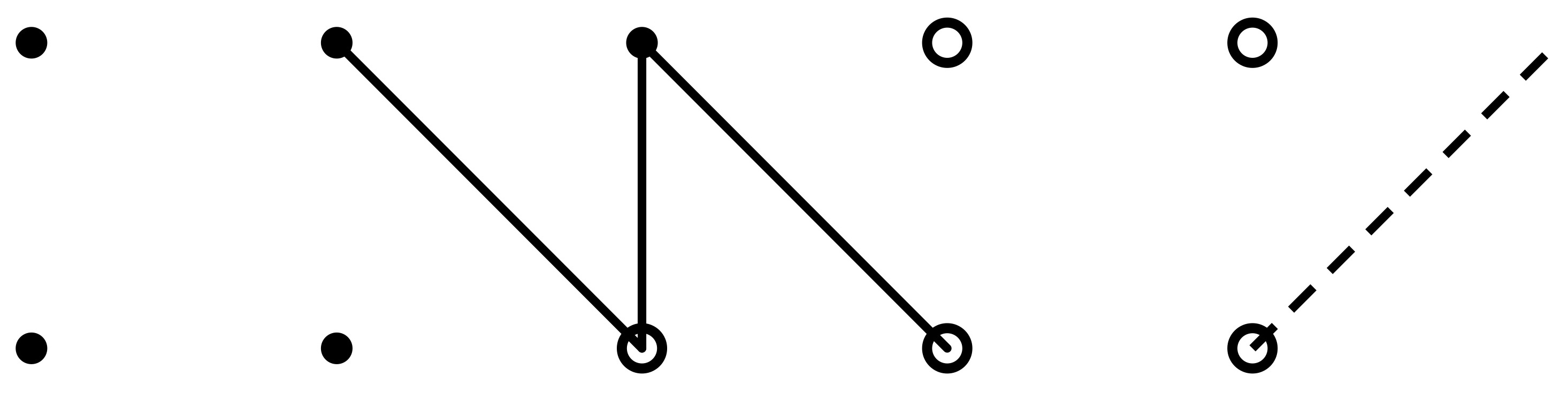}
\end{center} 
\caption{Vertices in $\ell$ to the left of the first gap (or first vertex in $\ell$ not in $S$) are shifted right by 1.}
\end{figure}

\end{enumerate}

Through a sequence of these steps, the vertices in $n$ and $\ell$ can be shifted, without increasing the boundary, to vertices in lines $n_0$ and $\ell_0$ respectively so that at least one of $n_0$ or $\ell_0$ is a segment:

\begin{align*}
\ell_0 &= \{(p, x \rightarrow i): a \leq x \leq b\} \\
& \quad \quad \quad \quad \text{OR} \\
n_0  &= \{(p+\epsilon, x \rightarrow i): c \leq x \leq d\} \\
\end{align*}

Finally, it is not hard to see that if one of $\ell_0$ or $n_0$ is a segment, the edge boundary can only stay the same or go down if the vertices in each of those lines are now centralized.  Thus, we have shown that
\begin{equation*}
 \left|  \partial_e(S_i,p,\epsilon)  \right| \leq  \left|  \partial_e(S,p,\epsilon)  \right|
\end{equation*}
which, by the arguments above, implies that 
\begin{equation*}
\left|\partial_e S_i\right| \leq \left| \partial_e S \right|
\end{equation*}

\end{proof}

%
%

\bibliographystyle{plain}

\end{document}